\documentclass[11pt]{article}
\usepackage[margin=1in]{geometry} 
\usepackage{amssymb,amsfonts,amsmath,bbm,mathrsfs,stmaryrd}
\usepackage{xcolor}
\usepackage{url}

\usepackage{enumerate}

\usepackage[colorlinks,
             linkcolor=black!75!red,
             citecolor=blue,
             pdftitle={},
             pdfproducer={pdfLaTeX},
             pdfpagemode=None,
             bookmarksopen=true
             bookmarksnumbered=true]{hyperref}

\usepackage{tikz}
\usetikzlibrary{arrows,calc,decorations.pathreplacing,decorations.markings,intersections,shapes.geometric,through,fit,shapes.symbols,positioning,decorations.pathmorphing}

\usepackage{braket}

\usepackage[amsmath,thmmarks,hyperref]{ntheorem}
\usepackage{cleveref}

\creflabelformat{enumi}{#2(#1)#3}

\crefname{section}{Section}{Sections}
\crefformat{section}{#2Section~#1#3} 
\Crefformat{section}{#2Section~#1#3} 

\crefname{appendix}{Appendix}{Appendices}
\crefformat{appendix}{#2Appendix~#1#3} 
\Crefformat{appendix}{#2Appendix~#1#3} 

\crefname{subsection}{\S}{\S\S}
\AtBeginDocument{%
  \crefformat{subsection}{#2\S#1#3}%
  \Crefformat{subsection}{#2\S#1#3}%
}

\theoremstyle{plain}

\newtheorem{lemma}{Lemma}[section]
\newtheorem{proposition}[lemma]{Proposition}
\newtheorem{corollary}[lemma]{Corollary}
\newtheorem{theorem}[lemma]{Theorem}

\theoremstyle{nonumberplain}

\theoremstyle{plain}
\theorembodyfont{\upshape}
\theoremsymbol{\ensuremath{\blacklozenge}}

\newtheorem{definition}[lemma]{Definition}
\newtheorem{example}[lemma]{Example}
\newtheorem{remark}[lemma]{Remark}

\newtheorem{notation}[lemma]{Notation}

\crefname{definition}{definition}{definitions}
\crefformat{definition}{#2definition~#1#3} 
\Crefformat{definition}{#2Definition~#1#3} 

\crefname{ex}{example}{examples}
\crefformat{example}{#2example~#1#3} 
\Crefformat{example}{#2Example~#1#3} 

\crefname{remark}{remark}{remarks}
\crefformat{remark}{#2remark~#1#3} 
\Crefformat{remark}{#2Remark~#1#3} 

\crefname{convention}{convention}{conventions}
\crefformat{convention}{#2convention~#1#3} 
\Crefformat{convention}{#2Convention~#1#3} 

\crefname{notation}{notation}{notations}
\crefformat{notation}{#2notation~#1#3} 
\Crefformat{notation}{#2Notation~#1#3} 

\crefname{table}{table}{tables}
\crefformat{table}{#2table~#1#3} 
\Crefformat{table}{#2Table~#1#3}

\crefname{lemma}{lemma}{lemmas}
\crefformat{lemma}{#2lemma~#1#3} 
\Crefformat{lemma}{#2Lemma~#1#3} 

\crefname{proposition}{proposition}{propositions}
\crefformat{proposition}{#2proposition~#1#3} 
\Crefformat{proposition}{#2Proposition~#1#3} 

\crefname{corollary}{corollary}{corollaries}
\crefformat{corollary}{#2corollary~#1#3} 
\Crefformat{corollary}{#2Corollary~#1#3} 

\crefname{theorem}{theorem}{theorems}
\crefformat{theorem}{#2theorem~#1#3} 
\Crefformat{theorem}{#2Theorem~#1#3} 

\crefname{enumi}{}{}
\crefformat{enumi}{(#2#1#3)}
\Crefformat{enumi}{(#2#1#3)}

\crefname{assumption}{assumption}{Assumptions}
\crefformat{assumption}{#2assumption~#1#3} 
\Crefformat{assumption}{#2Assumption~#1#3} 

\crefname{equation}{}{}
\crefformat{equation}{(#2#1#3)} 
\Crefformat{equation}{(#2#1#3)}

\numberwithin{equation}{section}

\theoremstyle{nonumberplain}
\theoremsymbol{\ensuremath{\blacksquare}}

\newtheorem{proof}{Proof}
\newcommand\pf[1]{\newtheorem{#1}{Proof of \Cref{#1}}}

\newcommand\bB{{\mathbb B}}
\newcommand\bC{{\mathbb C}}

\newcommand\bR{{\mathbb R}}
\newcommand\bS{{\mathbb S}}

\newcommand\bZ{{\mathbb Z}}

\newcommand\cA{{\mathcal A}}
\newcommand\cB{{\mathcal B}}

\newcommand\cD{{\mathcal D}}

\newcommand\cI{{\mathcal I}}
\newcommand\cJ{{\mathcal J}}
\newcommand\cK{{\mathcal K}}

\newcommand\cT{{\mathcal T}}

\newcommand{\cat}[1]{\textsc{#1}}

\newcommand{\qedhere}{\mbox{}\hfill\ensuremath{\blacksquare}}

\title{Rigidity results for automorphisms of Hardy-Toeplitz $C^*$-algebras}
\author{Alexandru Chirvasitu}

\begin{document}

\date{}

\newcommand{\Addresses}{{
  \bigskip
  \footnotesize

  \textsc{Department of Mathematics, University at Buffalo, Buffalo,
    NY 14260-2900, USA}\par\nopagebreak \textit{E-mail address}:
  \texttt{achirvas@buffalo.edu}

}}

\maketitle

\begin{abstract}
  We prove a number of results on the automorphisms of and isomorphisms between Hardy-Toeplitz algebras $\mathcal{T}(D)$ associated to bounded symmetric domains $D$: that the stable isomorphism class of $\mathcal{T}(D)$ determines $D$ (even when it is reducible), that for reducible domains $D=D_1\times\cdots \times D_s$ the automorphisms of the Shilov boundary $\check{S}(D)$ induced by those of $\mathcal{T}(D)$ permute the Shilov boundaries $\check{S}(D_i)$, and that by contrast to arbitrary solvable algebras, automorphisms of $\mathcal{T}(D)$ that are trivial on their character spaces $\check{S}(D)$ are trivial on the entire spectrum $\widehat{\mathcal{T}(D)}$.  
\end{abstract}


\vspace{.5cm}



\section*{Introduction}

Let $D$ be a bounded symmetric domain and $\check{S}$ its Bergman-Shilov boundary (see e.g. \cite{kw} or \Cref{subse:bdd} below for a recollection). We then have a Hardy Hilbert space $H^2:=H^2(\check{S})$ and a $C^*$-algebra $\cT(D)$ (the {\it Toeplitz algebra} of $D$) generated by the (Hardy-)Toeplitz operators $T_f$ on $H^2$ with continuous symbol $f\in C(\check{S})$ (\Cref{se:toeprel}).

The structure and representation theory of $\cT(D)$ has been studied extensively and is by now a very rich subject. For a necessarily woefully incomplete sampling of the literature the reader can consult for instance \cite{bc-wh,bck-wh,cob-sing,up0,up-alg} (concerned with elucidating the structure of $\cT(D)$) or \cite{up-bk} for a more comprehensive account, or \cite{up-surv} for a recent survey which in turn cites the ample literature.

In the present paper the obtain a number of ``rigidity'' results on $\cT(D)$ and their automorphisms. These come in several flavors, roughly within the same circle of ideas. One branch of the discussion focuses on recovering the domain $D$ from the $C^*$-algebra $\cT(D)$. The main result in this direction is that this is indeed possible:

\begin{theorem}[\Cref{th:stabiso,th:stabisogen}]
  Two bounded symmetric domains are isomorphic if and only if their Toeplitz algebras are stably isomorphic.
\end{theorem}

Secondly, we turn to automorphisms of $\cT(D)$ for possibly-reducible bounded symmetric domains
\begin{equation}\label{eq:ddis}
  D=D_1\times\cdots\times D_s.
\end{equation}
Such an automorphism will always induce a self-homeomorphism of the space
\begin{equation}\label{eq:shildec}
  \check{S}(D)=\check{S}(D_1)\times\cdots\times \check{S}(D_s)
\end{equation}
of characters of $\cT(D)$, and it is natural to ask which homeomorphisms arise in this fashion (or to put it differently, which homeomorphisms of $\check{S}(D)$ {\it lift} to $\cT(D)$). \cite{bdf,ys-sph}, for instance, answer the question for the open balls
\begin{equation*}
  D=\bB^{2n}\subset \bC^n
\end{equation*}
with respective Shilov boundaries $\check{S}(D)=\bS^{2n-1}$: the liftable homeomorphisms of the sphere are precisely those of degree 1. We do not quite answer the question here, as the problem is apparently still open for arbitrary {\it irreducible} bounded symmetric domains. The focus, rather, is on understanding how an automorphism of $\cT(D)$ plays with the tensor product decomposition
\begin{equation}\label{eq:tdtens}
  \cT(D)\cong \cT(D_1)\otimes\cdots\otimes \cT(D_s)
\end{equation}
corresponding to \Cref{eq:ddis}. Roughly speaking:
\begin{theorem}[\Cref{th:permtens,cor:1eq}]
  Let $\alpha$ be an automorphism of the Toeplitz algebra \Cref{eq:tdtens} corresponding to a bounded symmetric domain $D$ with a decomposition \Cref{eq:ddis} into irreducible factors.  
  
  The homeomorphism of the Shilov boundary $\check{S}(D)$ induced by $\alpha$ permutes the Cartesian factors $\check{S}(D_i)$ in the decomposition \Cref{eq:shildec}.
\end{theorem}

Finally, given an automorphism of $\cT(D)$, we look to the homeomorphism it induces on the spectrum $\widehat{\cT(D)}$. Here the main result is a kind of ``automatic triviality'' statement:
\begin{theorem}[\Cref{th:trivact}]
  If an automorphism of $\cT(D)$ acts trivially on the character space
  \begin{equation}\label{eq:inc}
    \check{S}(D)\subset \widehat{\cT(D)}
  \end{equation}
  then it acts trivially on the entire spectrum $\widehat{\cT(D)}$.
\end{theorem}
This presumably fits with the general rigidity theme, saying, as it does, that the homeomorphism of $\widehat{\cT}(D)$ induced by an automorphism of $\cT(D)$ is uniquely determined by its effect on just the compact Hausdorff layer \Cref{eq:inc} of the spectrum.

\Cref{se:prel} gathers background material on operator algebras, bounded symmetric domains, and so on, used in the subsequent discussion.

In \Cref{se:toeprel} we collect a number of general remarks on Toeplitz algebras and prove \Cref{th:stabiso}, which is the particular case of \Cref{th:stabisogen} applicable only to {\it irreducible} bounded symmetric domains.

\Cref{se:auto} is devoted to
\begin{itemize}
\item the proof of \Cref{th:permtens,cor:1eq}, on the fact that homeomorphisms of $\check{S}(D)$ liftable to automorphisms of $\cT(D)$ permute the Cartesian factors in \Cref{eq:shildec}, and separately
\item \Cref{th:stabisogen}, generalizing the domain-reconstruction \Cref{th:stabiso} to arbitrary domains.
\end{itemize}

In the short \Cref{se:autotriv} we prove that automorphisms of $\cT(D)$ induce homeomorphisms of $\widehat{\cT(D)}$ uniquely determined by their effect on just the Shilov boundary (\Cref{th:trivact}).

Finally, \Cref{se:a} recalls some tabular data on irreducible bounded symmetric domains, resulting from their classification and pertinent to the discussion.

\subsection*{Acknowledgements}

I am grateful for insightful comments and pointers to the literature by Jingbo Xia and Lewis Coburn.

This work is partially supported by NSF grant DMS-1801011.

\section{Preliminaries}\label{se:prel}

\subsection{Point-set topology}\label{subse:top}

The topological background is relatively slim, but we do need to recall some terminology (see for instance \cite[Chapter II, Exercise 3.17]{hart}).

\begin{definition}\label{def:spec}
  Let $x,y$ be two points in a topological space $X$. We say that
  \begin{itemize}
  \item $x$ {\it specializes to $y$}, or
  \item $y$ is a {\it specialization of $x$}, or
  \item $x$ is a {\it generization of $y$}
  \end{itemize}
  and write $x\to y$ if $y$ is contained in the closure of $\{x\}$.  
\end{definition}

Note that
\begin{equation*}
  x\ge y\iff x\to y
\end{equation*}
is a preorder on $X$ (i.e. a relation satisfying all of the properties of a partial order except perhaps for antisymmetry), and a partial order if the space $X$ is $T_0$ (meaning that for every two points there is an open set containing one but not the other \cite[\S 1.5]{eng}). Since we are interested in spectra of type-I $C^*$-algebras which are automatically $T_0$ (\cite[\S 3.1.3 and \S 3.1.6]{dix-cast}), we henceforth assume all of our topological spaces are $T_0$ unless specified otherwise.

\begin{definition}\label{def:chain}
  A {\it length-$k$ chain} in a topological space $X$ is a sequence of points 
  \begin{equation}\label{eq:chain}
    x_0\to x_1\to \cdots\to x_k
  \end{equation}
  with $x_i\ne x_{i+1}$ and the arrows indicating specialization as in \Cref{def:spec}.

  A {\it refinement} of \Cref{eq:chain} is a possibly-longer chain containing the points $x_i$ as some of its members. The chain is {\it maximal} if it does not admit a refinement.
\end{definition}

\subsection{$C^*$-algebras}\label{subse:cast}

We write $\cB(H)$ for bounded operators on a Hilbert space $H$, and $\cK(H)$ for the ideal of compact operators therein. Collectively, $\cK(H)$ are referred to as {\it elementary} $C^*$-algebras (e.g. \cite[\S 4.1.1]{dix-cast}). 

We take for granted much of the theory of $C^*$-algebra spectra (as treated, for instance, in \cite[Chapter 3]{dix-cast}), denoting the spectrum of $\cA$ by $\widehat{\cA}$. All $C^*$-algebras to which we apply the notion are type-I (or {\it postliminal} in the sense of \cite[Chapter 4]{dix-cast}, or {\it GCR} in the sense of \cite[Introduction]{kap}), so all reasonable versions of the spectrum agree (e.g. the set of primitive ideals and the set of isomorphism classes of irreducible representations) \cite[Theorem 4.3.7]{dix-cast}.

On occasion, we refer to {\it solvable} $C^*$-algebras, as introduced in \cite{dyn}. Briefly:
\begin{definition}\label{def:solv}
  A possibly-non-unital $C^*$-algebra $\cA$ is {\it solvable} if it admits a finite filtration
  \begin{equation*}
    \{0\}=\cI_{-1}\subset \cI_0\subset \cdots\subset \cI_r=\cA
  \end{equation*}
  by $C^*$ ideals such that
  \begin{itemize}
  \item each subquotient $\cI_j/\cI_{j-1}$ is isomorphic to $C_0(X_j)\otimes \cK(H_j)$ for some locally compact Hausdorff $X_j$ and Hilbert space $H_j$;
  \item the sequence $\dim H_j$ is non-increasing.
  \end{itemize}
  The minimal $r$ for which this is possible is the {\it length} of the solvable $C^*$-algebra $\cA$. 
\end{definition}

The following (presumably well-known) remark appears difficult to locate in the literature, so we state and prove it briefly here. Cf. \cite[Lemma 5]{ys-t2}, where the corresponding statement is proved for two specific ideals of the Toeplitz $C^*$-algebra $\cT(\bB^2\times \bB^2)$ (see \Cref{se:toeprel} for an explanation of the $\cT(\cdot)$ notation).

\begin{lemma}\label{le:sumclsd}
The algebraic sum of two $C^*$-ideals in a $C^*$-algebra is automatically closed, and hence again a $C^*$-ideal.   
\end{lemma}
\begin{proof}
  Let $\cA$ be a $C^*$-algebra and $\cI,\cJ\subseteq \cA$ two ideals. According to \cite[Theorem (2.4)]{lux} (whose proof is corrected in \cite[Remark 1.6 (b)]{voigt}) we have to argue that the annihilator
  \begin{equation*}
    (\cI\cap \cJ)^0 = \{\phi\in \cA^*\ \text{such that}\ \phi|_{\cI\cap \cJ}= 0\}
  \end{equation*}
  is the algebraic sum of the annihilators $\cI^0$ and $\cJ^0$. Now, the latter two annihilators are
  \begin{equation*}
    \cI^0\cong (\cA/\cI)^*\text{ and }\cJ^0\cong (\cA/\cJ)^*,
  \end{equation*}
  and similarly
  \begin{equation*}
    (\cI\cap \cJ)^0\cong (\cA/\cI\cap \cJ)^*.
  \end{equation*}
  Denoting $\pi_\cI:\cA\to \cA/\cI$ and similarly for $\cJ$, the $C^*$ morphism
  \begin{equation*}
 \begin{tikzpicture}[auto,baseline=(current  bounding  box.center)]
  \path[anchor=base] 
  (0,0) node (1) {$\cA$} 
  +(4,0) node (2) {$\cA/\cI\times \cA/\cJ$}
  ;
  \draw[->] (1) to[bend left=0] node[pos=.5,auto] {$\scriptstyle (\pi_\cI, \pi_\cJ)$} (2);
 \end{tikzpicture}
\end{equation*}
factors through an isometric embedding
\begin{equation*}
  \cA/(\cI\cap \cJ)\subseteq \cA/\cI\times \cA/\cJ. 
\end{equation*}
The dual of the right hand side, which is simply the direct sum of $\cI^0$ and $\cJ^0$, surjects by the Hahn-Banach theorem onto the dual $(\cI\cap \cJ)^0$ of the left hand side; this finishes the proof.
\end{proof}

\begin{remark}
  \Cref{le:sumclsd} also follows from Kober's criterion for the sum of two closed subspaces of a Banach space to be closed (\cite[Theorem 1]{kob} or \cite[Theorem (1.1)]{lux}) or from \cite[Corollary 1.4]{voigt} (the latter requiring that $\cI^0+\cJ^0$ be weak$^*$-closed).
\end{remark}

\subsection{Bounded symmetric domains}\label{subse:bdd} 

The topic is vast, and we recall only minuscule scattered fragments of it. The claims and results listed without an accompanying citation can be found in any of the abundant sources on symmetric spaces in general and Hermitian non-compact symmetric spaces in particular: \cite{helg} offers a comprehensive account, \cite{viv} is a briefer (and very accessible) survey, \cite{loos} draws the connection between bounded symmetric domains and Jordan theory, and the introductory sections of the various papers cited below (e.g. \cite{kw,up0,up-alg}) will often provide faster access to the theory in ready-to-use form.

Recall (e.g. \cite[\S VIII.7]{helg}):
\begin{definition}\label{def:bdd}
  A {\it bounded symmetric domain} is a bounded domain (open connected subset of some $\bC^n$) $D$ each of whose points is the unique fixed point of some involutive holomorphic automorphism.

  $D$ is {\it irreducible} if it does not decompose as a non-trivial Cartesian product of two other bounded symmetric domains.
\end{definition}

The classification of bounded symmetric domains is the cornerstone of the theory: they are (e.g. \cite[Chapter VIII, Theorems 6.1 and 7.1]{helg}) precisely the symmetric spaces of the form $G/K$ where
\begin{itemize}
\item $G$ is a non-compact, semisimple, connected and center-less Lie group;  
\item $K\subset G$ is a maximal compact subgroup (whose center will then automatically be a circle group; \cite[Chapter VIII, Proposition 6.2]{helg}).
\end{itemize}
The irreducible bounded symmetric domains were classified by \'E. Cartan \cite{cart-cls}. An alternative, pithy classification can be found in \cite{wolf-cls}, and the list appears in any number of sources: \cite[\S X.6.3 in conjunction with Table V in \S X.6.2]{helg}, \cite[Table 1]{viv}, \cite[\S 4]{loos}, etc.

The structure of the boundary $\partial D$ of a bounded symmetric domain is best analyzed in the context of studying the {\it Jordan triple system} attached to $D$, as we now recall briefly.

\begin{definition}\label{def:jts}
  A {\it Hermitian positive Jordan triple system} (JTS for short) is a complex vector space $V$ equipped with a ternary operation
  \begin{equation*}
    (x,y,z)\mapsto \{x,y,z\}=:L(x,y)z,
  \end{equation*}
  linear in the two outer variables and antilinear in the second, and satisfying the following additional conditions
  \begin{enumerate}[(1)]
  \item $\{x,y,z\}=\{z,y,x\}$;
  \item we have the identity
    \begin{equation*}
      [L(x,y),L(u,v)] = L(\{x,y,u\},v) - L(u,\{v,x,y\});
    \end{equation*}
  \item the Hermitian form $(x,y)\mapsto L(x,y)$ is positive definite.    
  \end{enumerate}
  We occasionally suppress the commas in expressions $\{x,y,z\}$.
\end{definition}
See for instance \cite[Definition 2.38]{viv} or \cite[\S 5.4]{up-surv}. We will not work directly with Jordan triple systems to any significant extent, so the phrase `Jordan triple system' can always be assumed here to entail `Hermitian' and `positive definite'.

For the following notion(s) see \cite[\S 3]{loos} (a lengthier account) or \cite[p.178]{viv} and \cite[p.151]{up-surv} for a brief recollection.
\begin{definition}\label{def:trip}
  A {\it tripotent} in a Jordan triple system $(V,\{\cdots\})$ is an element $e$ such that $\{eee\}=e$.

  Two tripotents $e$ and $f$ are {\it orthogonal} if $\{eef\}=0$ (this is in fact a symmetric condition).

  A tripotent is {\it primitive} if it does not decompose as a sum of two non-zero orthogonal tripotents.

  Finally, the {\it rank} of a tripotent $e$ is the largest number of summands in a decomposition of $e$ as a sum of mutually-orthogonal primitive tripotents.
\end{definition}
Tripotents are sometimes referred to as `idempotents' in the literature (e.g. \cite{amrt}), and the ternary operation is sometimes scaled so that the appropriate definition entails $\{eee\}=2e$ instead. 

Every non-zero element $x\in V$ (for a positive Hermitian JTS $(V,\{\cdots\})$) decomposes uniquely as
\begin{equation*}
  x=\lambda_1 e_1+\cdots+\lambda_s e_s
\end{equation*}
where
\begin{equation*}
  0<\lambda_1<\cdots<\lambda_s
\end{equation*}
and the $e_i$ are mutually orthogonal tripotents (\cite[Corollary 3.12]{loos}). Associating $\lambda_s$ to $x$ produces the {\it spectral norm} on $V$, Finally, to circle back to bounded symmetric domains, the foundational result (\cite[Theorem 4.1]{loos}, due originally to Koecher \cite{koe}) is that bounded symmetric domains can be realized essentially uniquely as spectral-norm open unit balls in (positive Hermitian) Jordan triple systems.

\begin{notation}
  For a bounded symmetric domain $D$ we write $V(D)$ (or simply $V$ when $D$ is implicit) for the underlying complex vector space of the JTS attached to $D$.
\end{notation}

Henceforth, whenever working with a bounded symmetric domain $D$, we assume it realized as the spectral open unit ball of its associated JTS $V(D)$. The compact group $K$ in the realization $D\cong G/K$ discussed above acts as a group of automorphisms of $V(D)$ as a Jordan triple system (i.e. preserves all of the structure).

With this in place, one possible approach (e.g. \cite[\S 5.4]{loos}) to the rank of a bounded symmetric domain is
\begin{definition}\label{def:rk}
  Let $D$ be a bounded symmetric domain.  The {\it rank} of $D$ is the maximal rank of a tripotent in $V(D)$ in the sense of \Cref{def:trip}. 
\end{definition}

Let $D$ be a bounded symmetric domain of rank $r$. For each $0\le k\le r$, the space $T_k\subset V$ of rank-$k$ tripotents is a smooth manifold contained in $\partial D$, invariant under the action of $K$, and connected and homogeneous under that $K$-action when $D$ is irreducible \cite[Corollary 5.12]{loos}.

\begin{definition}\label{def:shil}
  The {\it Shilov (or Bergman-Shilov) boundary} $\check{S}(D)$ of a bounded symmetric domain $D$ of rank $r$ is its manifold $T_r\subset \partial D$ of maximal-rank tripotents.
\end{definition}
See also \cite[Theorem 6.5]{loos} for alternative characterizations of $\check{S}(D)$. It turns out that (as the name suggests) it is indeed a Shilov boundary in the functional-analytic sense \cite[Exercise 2.27]{doug-bk}: it is the smallest closed subset of $\partial D$ on which every function continuous on $\overline{D}$ and holomorphic on $D$ achieves its maximum. Note that $\check{S}(D)$ is {\it always} homogeneous under the $K$ action, regardless of whether $D$ is irreducible \cite[Theorem 5.3]{loos}.

\section{Toeplitz algebras}\label{se:toeprel}

Let $D$ be a bounded symmetric domain. By the homogeneity of $\check{S}:=\check{S}(D)$ under the action of the compact group $K$, there is a unique $K$-invariant probability measure $\mu$ on $\check{S}$. This affords a Hilbert space $L^2(\check{S}):=L^2(\check{S},\mu)$ (we will typically omit $\mu$), as well as

\begin{definition}\label{def:hardy}
  The {\it Hardy space} $H^2(\check{S}(D))$ is the closure in $L^2(\check{S}(D))$ of the space of (restrictions to $\check{S}(D)$ of) polynomials on $V:=V(D)$.

  The {\it Szeg\"o projection} (associated to all of this data) is the orthogonal projection $P:L^2(\check{S})\to H^2(\check{S})$.
\end{definition}
This is one possible definition for Hardy spaces (see e.g. \cite[\S 1]{up0}, which in turn cites \cite[\S 4]{kor-poi}). 
\begin{definition}\label{def:toepalg}
  Let $D$ be a bounded symmetric domain and $f\in C(\check{S})$ a continuous function on its Shilov boundary $\check{S}:=\check{S}(D)$. The {\it Toeplitz operator} $T_f$ on $H^2(\check{S})$ is $PM_f$, where
  \begin{itemize}
  \item $P:L^2(\check{S})\to H^2(\check{S})$ is the Szeg\"o projection as in \Cref{def:hardy}, and
  \item $M_f$ is (the restriction to $H^2(\check{S})\subset L^2(\check{S})$ of) the multiplication-by-$f$ operator.
  \end{itemize}
  $f$ is the {\it symbol} of the associated Toeplitz operator $T_f$.  
  
  The {\it Toeplitz (or Hardy-Toeplitz) algebra} $\cT(D)$ attached to $D$ is the $C^*$-subalgebra of $\cB(H^2(\check{S}))$ generated by all $T_f$ for continuous $f\in C(\check{S})$. 
\end{definition}
For structure theory on $\cT(D)$ the reader can consult \cite{up0,up-alg} as well as the earlier sources \cite{bc-wh,bck-wh,cob-sing}, which treat particular cases of bounded symmetric domains.


According to \cite[Theorem 3.12]{up-alg} $\cT(D)$ admits a filtration
\begin{equation}\label{eq:filt}
  \{0\}=\cI_{-1}\subset \cI_0\subset \cdots\subset \cI_r=\cT(D)
\end{equation}
where
\begin{itemize}
\item $r$ is the rank of the symmetric domain $D$;
\item each subquotient $\cI_s/\cI_{s-1}$, $s\in 0..r$ is isomorphic to $C(T_s)\otimes \cK(H_s)$ for some Hilbert space $H_s$, where $T_s\subset \partial D$ is the compact manifold of rank-$s$ tripotents in $V$; 
\item $\cI_0=\cK(H)$ for the main Hardy space $H=H^2(\check{S}(D))$ on which $\cT(D)$ is realized faithfully.
\end{itemize}

Note that our numbering for the ideals $\cI_j$ is shifted as compared to \cite{up-alg}: in the latter $\cI_0$ is the trivial ideal, whereas here it is $\cI_{-1}$ that vanishes.

\begin{notation}\label{not:trip}
  For a bounded symmetric domain $D$ we write $T_{s,D}$ or $T_s(D)$ for the compact manifold of rank-$s$ tripotents in the Jordan triple system attached to $D$.

  Similarly, we write $\cI_{j,D}$ or $\cI_j(D)$ for the ideals $\cI_j$ appearing in the filtration \Cref{eq:filt}. 
\end{notation}

\begin{remark}\label{re:i0char}
  It follows from \cite[Corollaries 3.11 and 3.13]{up-alg} that $\cI_0$ is a {\it characteristic} ideal, i.e. preserved by all automorphisms of $\cT(D)$. Indeed, those results ensure that $\cI_0$ is precisely the common kernel of all non-faithful irreducible representations of $\cT(D)$.

  In fact, we will see in \Cref{pr:char} that the ideals \Cref{eq:filt} are all characteristic. 
\end{remark}

\begin{proposition}\label{pr:char}
  Let $D$ be a rank-$r$ bounded symmetric domain and $\cK$ an elementary $C^*$-algebra. Then, the filtration of $\cT(D)\otimes \cK$ obtained by tensoring \Cref{eq:filt} with $\cK$ is characteristic, in the sense that $\cI_j\otimes \cK$ is preserved by every automorphism of $\cT(D)\otimes \cK$. 
\end{proposition}
\begin{proof}
  To simplify matters we ignore $\cK$ and work with $\cT:=\cT(D)$ throughout, but the arguments transport verbatim, via the general remark that for a $C^*$-algebra $\cA$ the irreducible representations of $\cA\otimes \cK(H)$ are precisely
  \begin{equation*}
    \pi\otimes \rho:\cA\otimes \cK\to \cB(H_{\pi})\otimes \cB(H),
  \end{equation*}
  where $\pi:\cA\to \cB(H_{\pi})$ is an irreducible $\cA$-representation and $\rho$ is the unique irreducible representation of $\cK=\cK(H)$.

  Let $\alpha$ be an automorphism of $\cT$, inducing a self-homeomorphism (denoted abusively by the same symbol) of the spectrum $\widehat{\cT}$.

  It follows from \cite[Lemma 3.4]{up-alg} that for $k\in 0..r$ the elements of the locally closed subset
  \begin{equation}\label{eq:kk-1}
    \widehat{\cI_k/\cI_{k-1}}\subseteq \widehat{\cT}
  \end{equation}
  are precisely those points $x\in \widehat{\cT}$ that fit into a maximal chain
  \begin{equation*}
    x_0\to\cdots\to x_k=x
  \end{equation*}
  in the sense of \Cref{def:chain}, and hence the open subset 
  \begin{equation*}
    \widehat{\cI_k}\subseteq \widehat{\cT}
  \end{equation*}
  consists of the targets $x$ of maximal chains 
  \begin{equation*}
    x_0\to\cdots\to x_k'=x\text{ for some }k'\le k.
  \end{equation*}
  The property of fitting into such a chain is purely topological in nature, and hence invariant under $\alpha$. Since
  \begin{equation*}
    \cI\mapsto \widehat{\cI}
  \end{equation*}
  is a bijection between ideals of a $C^*$-algebra and open subsets of its spectrum (\cite[\S 3.2.2]{dix-cast}), the $\alpha$-invariance of $\widehat{\cI_k}$ entails that of $\cI_k$.
\end{proof}


The characterization of the points in \Cref{eq:kk-1} via chains used in the proof of \Cref{pr:char} also gives the following procedure for reconstructing $r$ from $\cT(D)$ via the notion of length of a solvable $C^*$-algebra in the sense of \Cref{def:solv}.

\begin{lemma}\label{le:islng}
  For a rank-$r$ bounded symmetric domain $D$ and an elementary algebra $\cK$ the length of the solvable $C^*$-algebra $\cT(D)\otimes \cK$ is $r$.
  \qedhere
\end{lemma}

\begin{remark}
  The length-$r$ claim is part of \cite[Theorem 4.11.133]{up-bk}, but as far as I can tell it is not actually argued there that the length is $r$ rather than smaller. This is in principle conceivable, since the existence of a filtration as in \Cref{def:solv} does not necessarily preclude a {\it shorter} filtration of that form.
\end{remark}

\subsection{Reconstructing irreducible domains}\label{subse:recirr}

Below, we will repeatedly encounter isomorphisms
\begin{equation}\label{eq:stabiso}
  \cA\otimes \cK\cong \cB\otimes \cK,
\end{equation}
where
\begin{itemize}
\item $\cA$ and $\cB$ are Toeplitz algebras (attached to two bounded symmetric domains, a priori different) and
\item $\cK=\cK(H)$ is the algebra of compact operators on an $\aleph_0$-dimensional Hilbert space $H$. 
\end{itemize}
As is customary (e.g. \cite[p.337]{br-stab}), we refer to \Cref{eq:stabiso} as a {\it stable isomorphism} between $\cA$ and $\cB$.

\begin{theorem}\label{th:stabiso}
  If $D_1$ and $D_2$ are two irreducible bounded symmetric domains such that $\cT(D_i)$ are stably isomorphic then $D_i$ are isomorphic.
\end{theorem}

We prove a series of auxiliary results.

\begin{lemma}\label{le:samerk}
  Under the hypotheses of \Cref{th:stabiso}
  \begin{enumerate}[(1)]
  \item\label{item:1} the domains $D_i$ have the same rank $r$;
  \item\label{item:2} for $s\in 1..r$ we have homeomorphisms $T_{s}(D_1)\cong T_s(D_2)$ for $T_s$ defined in \Cref{not:trip}.
  \item\label{item:3} $D_i$ have isomorphic Shilov boundaries $\check{S}(D_i)$. 
  \end{enumerate}
\end{lemma}
\begin{proof}
  \Cref{item:1} is immediate from \Cref{le:islng}. As for \Cref{item:2}, it follows from the argument in the proof of \Cref{pr:char}, giving a purely topological characterization of 
  \begin{equation*}
    T_{s}(D)\cong \widehat{\cI_s(D)/\cI_{s-1}(D)}\subset \widehat{\cT(D)}
  \end{equation*}
  (see \Cref{not:trip}).

  Finally, \Cref{item:3} is a consequence of \Cref{item:1} and \Cref{item:2}:
  \begin{equation*}
    \check{S}(D_1) = T_r(D_1)\cong T_r(D_2) = \check{S}(D_2).
  \end{equation*}
\end{proof}


\begin{lemma}\label{le:samedim}
Under the hypotheses of \Cref{th:stabiso} $D_i$ have the same dimension. 
\end{lemma}
\begin{proof}
  The proof will be by induction on the (common, by \Cref{le:samerk}) rank of $D_i$.

  {\bf Base case: rank 1.} The rank-1 domains are exactly those customarily labeled as type-$I_{p,1}$, $p\ge 1$ (e.g. \cite[Tables 1 and 7]{viv}), i.e. the open unit balls $\bB^{2p}$ in $\bC^p$ and their Shilov boundaries are the spheres $\bS^{2p-1}$. In particular the rank-1 domains are uniquely determined by their Shilov boundaries, which by \Cref{le:samerk} invariants of the (stable) Toeplitz algebras.

  {\bf Induction step.} For an irreducible boundary symmetric domain $D$ of rank $r>1$ it follows from \cite[Theorem 6.3]{loos} that the boundary of $D$ decomposes as a disjoint union of strata
  \begin{equation*}
    \cat{str}_k:=\cat{str}_k(D),\ k= 1..r
  \end{equation*}
  which fiber, respectively, over the manifolds $T_k(D)$ of rank-$k$ tripotents (see also \cite[Remark 6.9]{loos}). The fiber over a tripotent $e$ of rank $k$ is the bounded symmetric domain $D_e=D\cap V_0(e)$, where $V_0(e)$ is the annihilator of a linear operator associated to $e$. In particular we have
  \begin{equation}\label{eq:dimex}
    \dim D = \dim \partial D + 1 = \max_{k=1}^r~ (\dim T_k(D)+\dim D_e),
  \end{equation}
  where
  \begin{itemize}
  \item `$\dim$' denotes the covering dimension (e.g. the usual notion of dimension for manifolds) and
  \item $e$ on the right hand side is an arbitrary rank-$k$ tripotent, since the dimension of $D_e$ does not depend on the choice by the homogeneity of $T_k(D)$ under the automorphism group of $D$ (\cite[Corollary 5.12]{loos}).
  \end{itemize}
  
  As in the proof of \Cref{pr:char}, nothing below is affected by passing to stabilizations, so we work with an isomorphism $\alpha$ between $\cT_i:= \cT(D_i)$ instead of $\cT_i\otimes \cK$ in order to avoid overburdening the notation. By \Cref{eq:dimex} it will suffice to show that for each $k= 1..r$ we have an isomorphism
  \begin{equation*}
    D_{1,e_1}\cong D_{2,e_2}
  \end{equation*}
  for some rank-$k$ tripotents $e_i$ in the Jordan systems $V_i$ attached respectively to $D_i$. Moreover, because $D_e$ is simply a point when $e$ has maximal rank $r$, it is enough to do this for $1\le k\le r-1$. 

  Throughout the rest of the proof fix $1\le k\le r-1$ and a rank-$k$ tripotent $e_1\in V_1$, parametrizing an irreducible representation $\rho_1:\cT_1\to \cB(H^2(\check{S}(D_{e_1})))$ (\cite[discussion following Theorem 3.3]{up-alg}) whose class belongs to
  \begin{equation*}
    \widehat{\cI_k(D_1)/\cI_{k-1}(D_2)}\subset \widehat{\cT_1}. 
  \end{equation*}
  As argued in the proof of \Cref{le:samerk}, $\alpha$ transforms $\rho_1$ into an irreducible representation
  \begin{equation*}
    \rho_2:\cT_2\to \cB(H)
  \end{equation*}
  parametrized by some point in 
  \begin{equation*}
    \widehat{\cI_k(D_2)/\cI_{k-1}(D_2)}\subset \widehat{\cT_2},
  \end{equation*}
  i.e. a rank-$k$ tripotent $e_2\in V_2$. $\alpha$ will also identify the images $\rho_i(\cT_i)$ through these representations which, as explained on \cite[p.571]{up-alg}, are respectively isomorphic to $\cT(D_{i,e_i})$. In order to apply the induction hypothesis to obtain the desired isomorphism
  \begin{equation*}
    D_{1,e_1}\cong D_{2,e_2}
  \end{equation*}
  it suffices to observe that by \cite[Theorem 5.9]{loos}, if $D$ is irreducible then so is $D_e$ (for an arbitrary tripotent).
\end{proof}


\pf{th:stabiso}
\begin{th:stabiso}
  \Cref{le:samerk,le:samedim} show that the domains $D_i$ have the same rank and dimension and have isomorphic Shilov boundaries. By \Cref{cor:complinv} below, this is sufficient to identify an irreducible domain.
\end{th:stabiso}

\section{Automorphisms of tensor products}\label{se:auto}

Throughout the present section we work with a rank-$r$ bounded symmetric domain $D$, decomposable as a product
\begin{equation*}
  D=D_1\times\cdots\times D_s
\end{equation*}
of irreducible bounded symmetric domains $D_i$ of respective ranks $r_i$ (so that $r=\sum r_i$). We are interested in the automorphisms of the $C^*$-algebra
\begin{equation*}
  \cT:=\cT(D)\cong \cT(D_1)\otimes \cdots\otimes \cT(D_s).
\end{equation*}


Such an automorphism of $\cT$ will induce one of the Shilov boundary
\begin{equation}\label{eq:bdryprod}
  \check{S}(D)\cong \check{S}(D_1)\times \cdots\times \check{S}(D_s),
\end{equation}
since the latter is the spectrum of the top quotient $\cT(D)/\cI_{r-1}$ in the filtration \Cref{eq:filt}. Precisely which automorphisms of $\check{S}(D)$ arise in this manner is the subject of some prior work:
\begin{itemize}
\item The case of the unit disk $D=\bB^2$ in the complex plane is treated in \cite{bdf}, the conclusion being that the resulting automorphisms of $\check{S}(D)\cong \bS^1$ are exactly the orientation-preserving homeomorphisms of the circle. 
\item This is generalized in \cite{ys-sph} to unit balls $D=\bB^{2n}$ in $\bC^n$. Once more, the corresponding homeomorphisms of $\check{S}(D)\cong \bS^{2n-1}$ are the orientation-preserving ones (i.e. the degree-1 homeomorphisms). 
\item The pattern recurs in \cite{ys-t2}, where $D=\bB^2\times \bB^2$ is the Cartesian product of two unit disks. Here, the automorphisms of $\check{S}(D)\cong \bS^1\times \bS^1$ induced by those of $\cT(D)$ are the most straightforward imaginable:
  \begin{enumerate}[(a)]
  \item Cartesian products of orientation-preserving of the two $\bS^1$ factors, and
  \item compositions of these with the flip automorphism interchanging the two $\bS^1$ factors.
  \end{enumerate}
\end{itemize}

This latter result in \cite{ys-t2}, in particular, exemplifies a kind of rigidity phenomenon, whereby the automorphism induced on the Shilov boundary \Cref{eq:bdryprod} by one of the Toeplitz algebra simply permutes some mutually isomorphic Cartesian factors. To make the statement precise we need some notation and terminology.

\begin{definition}\label{def:sym}
  Denote by
  \begin{equation}\label{eq:fam}
    \cD:=(D_j,\ j= 1..s)
  \end{equation}
  a(n ordered) family of bounded symmetric domains. Its {\it symmetric group}
  \begin{equation*}
    \cat{sym}_{\cD}\subseteq \cat{sym}_s
  \end{equation*}
  is the permutation group on $1..r$ that permutes those $D_i$ which are mutually isomorphic.

  For a fixed bounded symmetric domain we denote by $\cat{ho}_*(D)$ the group of self-homeomorphisms of its Shilov boundary $\check{S}(D)$ descended from automorphisms of the Toeplitz algebra $\cT(D)$ and by $\cat{ho}_{*s}(D)$ those descended from automorphisms of the {\it stabilized} Toeplitz algebra $\cT(D)\otimes \cK$, where
  \begin{equation*}
    \cK:=\cK(H),\quad \dim H=\aleph_0. 
  \end{equation*}
  For a family \Cref{eq:fam} we similarly define
  \begin{align*}
    \cat{ho}_*(\cD)&:=\cat{ho}_*(D_1)\times \cdots\times \cat{ho}_*(D_s)\\
    \cat{ho}_{*s}(\cD)&:=\cat{ho}_{*s}(D_1)\times \cdots\times \cat{ho}_{*s}(D_s),
  \end{align*}
  and adopt the same conventions for full groups of self-homeomorphisms of Shilov boundaries, denoted simply by $\cat{ho}_{sh}$:
  \begin{equation*}
    \cat{ho}_{sh}(D) := \text{self-homeomorphisms of the Shilov boundary }\check{S}(D),
  \end{equation*}
  similarly
  \begin{equation*}
    \cat{ho}_{sh}(\cD) := \cat{ho}_{sh}(D_1)\times\cdots\times \cat{ho}_{sh}(D_s),
  \end{equation*}
  etc.
\end{definition}

\begin{remark}
  For a family \Cref{eq:fam} set
  \begin{equation*}
    D:=\prod_j D_j. 
  \end{equation*}
  The symmetric group $\cat{sym}(\cD)$ can be regarded as a subgroup of $\cat{ho}_{sh}(D)$, by permuting the respective Cartesian factors. Conjugation by $\cat{sym}_{\cD}$ in $\cat{ho}_{sh}(D)$ leaves $\cat{ho}_*(\cD)$ invariant and the two subgroups
  \begin{equation*}
    \cat{sym}_{\cD}\text{ and }\cat{ho}_*(\cD) \subseteq \cat{ho}_{sh}(D)
  \end{equation*}
  intersect trivially, so their product in the larger ambient group is isomorphic to the semidirect product
  \begin{equation*}
    \cat{ho}_*(\cD)\rtimes \cat{sym}(\cD)
  \end{equation*}
  associated to the conjugation action. 
\end{remark}

This is sufficient to make sense of the following statement, generalizing the computation in \cite[Theorems 3 and 4]{ys-t2} of the automorphism group of $\cT(\bB^2\times \bB^2)$ in terms of $\mathrm{Aut}~ \cT(\bB^2)$.

\begin{theorem}\label{th:permtens}
    Let $\cD$ be a family of irreducible bounded symmetric domains \Cref{eq:fam} and
  \begin{equation}\label{eq:dprod}
    D:=D_1\times\cdots\times D_s
  \end{equation}
  their product. We then have embeddings
  \begin{equation}\label{eq:dec}
    \cat{ho}_*(\cD)\rtimes \cat{sym}_{\cD}\subseteq \cat{ho}_*(D) \subseteq \cat{ho}_{*s}(\cD)\rtimes \cat{sym}_{\cD}. 
  \end{equation}
\end{theorem}

Since the filtration \Cref{eq:filt} of $\cT$ is obtained by tensoring those of $\cT_i$, we introduce more notation for the bookkeeping of the resulting ideals. Specifically, for a tuple
\begin{equation*}
  {\bf i} = (i_1,\cdots,i_s),\quad i_j= 0..r_j
\end{equation*}
we write
\begin{equation*}
  \cI_{\bf i} = \cI_{\bf i}(D):=\cI_{i_1}(D_1)\otimes\cdots\otimes \cI_{i_s}(D_s),
\end{equation*}
occasionally expanding the tuple as a subscript; e.g. $\cI_{i_1,\cdots,i_s}$. This is an ideal of $\cT$ and solvable of length
\begin{equation*}
  |{\bf i}|:=\sum_{j=1}^s i_j
\end{equation*}
as a $C^*$-algebra, and the length-$k$ ideal $\cI_k(D)$ of $\cT$ decomposes as 
\begin{equation}\label{eq:dekik}
  \cI_k(D) = \sum_{|{\bf i}|=k}\cI_{\bf i}(D). 
\end{equation}
Note that it is enough to take a sum (rather than a closed span), since the plain algebraic sum is already closed by \Cref{le:sumclsd}. We refer to $|{\bf i}|$ as the {\it weight} of the tuple (since `length' might be confused with the length $s$, i.e. the number of entries of ${\bf i}$).

\begin{lemma}\label{le:pressum}
  An automorphism of $\cT:=\cT(D)$ permutes the summands in the decomposition \Cref{eq:dekik} for every $k=1..r$. 
\end{lemma}
\begin{proof}
  We know from \Cref{pr:char} that an automorphism $\alpha$ preserves all $\cI_k:=\cI_k(D)$. Recall that argument: the points $x\in \widehat{\cI_k}\subseteq \widehat{\cT}$ are exactly the terminus points of maximal chains
  \begin{equation*}
    x_0\to \cdots\to x_{k'}=x\text{ for some }k'\le k. 
  \end{equation*}
  The conclusion now follows by induction on $k$, noting that for $|{\bf i}|=k$ the subsets
  \begin{equation*}
    \widehat{\cI_{\bf i}/\cI_{\bf i}\cap \cI_{k-1}}\subset \widehat{\cI_k/\cI_{k-1}}
  \end{equation*}
  are precisely the connected components of $\widehat{\cI_k/\cI_{k-1}}$, and hence are permuted by $\alpha$ (or rather by the action it induces on $\widehat{\cT}$).
\end{proof}

Of particular importance will be the tuples of respective weight $r_j$, $j= 1..s$, with all of that weight concentrated in position $j$; we thus need shorthand notation for those:
\begin{equation*}
\bullet_j:=(0,\cdots,0,r_j,0,\cdots,0),\quad j= 1..s
\end{equation*}
where the $r_j$ entry is the $j^{th}$. More generally, we write $\bullet_j^k$ for the tuple whose
\begin{itemize}
\item $j^{th}$ entry is $k$
\item all other entries vanish.
\end{itemize}
In particular, we have $\bullet_j=\bullet_j^{r_j}$.

\begin{lemma}\label{le:prestall}
  For each
  \begin{equation*}
    0\le k\le \max~r_i
  \end{equation*}
  an automorphism $\alpha$ of $\cT:=\cT(D)$ permutes the ideals $\cI_{\bullet_j^k}:=\cI_{\bullet_j^k}(D)$ for $j$ ranging over $1..s$.
\end{lemma}
\begin{proof}
  We argue by induction on $k$.

  {\bf Base case: $k=0$.} For all $j=1..s$ the ideal $\bullet_j^0$ coincides with $\cI_0(D)$, which is characteristic by \Cref{pr:char}. 

  {\bf Base case: $k=1$.} The sets
  \begin{equation*}
    \widehat{\bullet_j^1/\cI_0(D)}\subset \widehat{\cT},\ j=1..s
  \end{equation*}
  are precisely the connected components of $\widehat{\cI_1(D)/\cI_0(D)}$, and hence are permuted by $\alpha$.
  
  {\bf Inductive step.} The points 
  \begin{equation}\label{eq:layers}
    x\in \widehat{\bullet_j^k/\bullet_j^{k-1}}\subset \widehat{\cT}
  \end{equation}
  are those fitting into a maximal chain
  \begin{equation*}
    x_0\to \cdots\to x_k=x
  \end{equation*}
  with
  \begin{equation*}
    x_{\ell}\in \widehat{\bullet_j^{\ell}/\bullet_j^{\ell-1}},\ \forall \ell=0..k-1. 
  \end{equation*}
  By the inductive hypothesis, this characterization shows that the sets \Cref{eq:layers} are permuted by $\alpha$.
\end{proof}

In particular, we obtain

\begin{corollary}\label{cor:presfull}
  An automorphism of $\cT(D)$ permutes the ideals $\bullet_j$.
\end{corollary}
\begin{proof}
  Indeed, these are by definition
  \begin{equation*}
    \bullet_j=\bullet_j^{r_j}\cong \cT(D_j)\otimes K
  \end{equation*}
  for elementary $C^*$-algebras $\cK$ (depending on $j$). By \Cref{le:prestall} an automorphism will permute those with equal lengths $r_j$.  
\end{proof}

\pf{th:permtens}
\begin{th:permtens}
  We prove the two inclusions separately.
  
  {\bf Step 1: We have an embedding}
    \begin{equation*}
      \cat{ho}_*(\cD)\rtimes \cat{sym}_{\cD}\subseteq \cat{ho}_*(D).
    \end{equation*}
    The two semidirect factors embed as follows:
  \begin{itemize}
  \item An element of $\cat{sym}(\cD)$ lifts to an automorphism of the Toeplitz algebra
    \begin{equation*}
      \cT(D)\cong \cT(D_1)\otimes \cdots\otimes \cT(D_s)
    \end{equation*}
    by permuting the respective tensorands, and
  \item An element of
    \begin{equation*}
      \cat{ho}_*(\cD)=\prod_{i=1}^r \cat{ho}_*(D_j)
    \end{equation*}
    consists by definition of a tuple of automorphisms of $\check{S}(D_j)$ liftable respectively to automorphisms of $\cT(D_j)$. Lifting all $r$ of them produces an automorphism of the tensor product $\cT(D)$ of all $\cT(D_j)$.
  \end{itemize}

  {\bf Step 2: The opposite inclusion}
  \begin{equation*}
    \cat{ho}_*(D)\subseteq \cat{ho}_{*s}(\cD)\rtimes \cat{sym}_{\cD}.
  \end{equation*}  
  Let $\alpha$ be an automorphism of $\cT:=\cT(D)$.  By \Cref{cor:presfull} it permutes the ideals
  \begin{equation}\label{eq:bullj}
    \bullet_j\cong \cT_j\otimes \cK_j
  \end{equation}
  where
  \begin{itemize}
  \item $\cT_j:=\cT(D_j)$, and
  \item $\cK_j$ are elementary $C^*$-algebras. 
  \end{itemize}
  By \Cref{th:stabiso} those $\bullet_j$ that are permuted by $\alpha$ must correspond to isomorphic bounded symmetric domains $D_j$, so the permutation $\sigma_{\alpha}$ induced by $\alpha$ on the $\check{S}(D_j)$ belongs to $\cat{sym}_{\cD}$. Realizing $\sigma_{\alpha}$ as an automorphism group pf $\cT$ which permutes the respective tensorands, we can compose $\alpha$ with the inverse $\sigma_{\alpha}^{-1}$ and henceforth assume without loss of generality that $\alpha$ leaves each $\bullet_j$ invariant.

  Under that assumption $\alpha$ induces individual automorphisms
  \begin{equation*}
    \alpha_j\in \cat{ho}_{*s}(D_j),
  \end{equation*}
  since by construction they lift to automorphisms $\alpha|_{\bullet_j}$ of the ideals \Cref{eq:bullj}. Proceeding with the proof, we focus on $\alpha_1$ to fix ideas. The ideal
  \begin{equation}\label{eq:1ideal}
    \bullet_1=\cT_1\otimes \cI_0(D_2)\otimes \cdots \otimes \cI_0(D_s)\subset \cT
  \end{equation}
  is isomorphic to $\cT_1\otimes \cK$, so its spectrum can be identified with
  \begin{equation*}
    \widehat{\cT_1}\cong \{pt\}\sqcup \check{S}(D_1). 
  \end{equation*}
  $\alpha$ leaves \Cref{eq:1ideal} invariant, and $\alpha_i$ can be identified with the restriction of $\alpha$ to the locally closed subset
  \begin{equation}\label{eq:incd1}
    \check{S}(D_1)\subset \widehat{\cT_1}\subset \widehat{\cT}. 
  \end{equation}
  Writing
  \begin{equation*}
    D_1':=D_2\times \cdots\times D_s
  \end{equation*}
  and denoting by $pt_1'$ the open point of the spectrum of $\cT(D_1')$, the inclusion \Cref{eq:incd1} identifies $\check{S}(D_1)$ with
  \begin{equation}\label{eq:d1pt}
    \check{S}(D_1)\times \{pt_1'\}\subset \prod_{j=1}^s\widehat{\cT_j}, 
  \end{equation}
  while the Shilov boundary $\check{S}(D)$ decomposes as
  \begin{equation}\label{eq:d1d1'}
    \check{S}(D) = \check{S}(D_1)\times \check{S}(D_1'). 
  \end{equation}
  Since $\alpha$ restricts to $\alpha_1$ on \Cref{eq:d1pt} and $\{pt_1'\}$ is dense in $\widehat{\cT(D_1')}$, the decomposition \Cref{eq:d1d1'} means that $\alpha$ restricts to $\check{S}(D)$ as
  \begin{equation*}
    \alpha_1\times (\text{some homeomorphism on }\check{S}(D_1')). 
\end{equation*}
  Applying this argument to the other Cartesian factors of $D$, we conclude that
  \begin{equation*}
    \alpha|_{\check{S}(D)} = \alpha_1\times\cdots\alpha_r.
  \end{equation*}
  This finishes the proof of the theorem.
\end{th:permtens}

We will be able to say more when all $D_j$ are of rank 1. First, we have

\begin{lemma}\label{le:1easy}
  For a rank-1 bounded symmetric domain $D$ we have
  \begin{equation*}
    \cat{ho}_*(D) = \cat{ho}_{*s}(D),
  \end{equation*}
  i.e. the automorphisms induced on $\check{S}(D)$ by those of $\cT:=\cT(D)$ and those of $\cT\otimes \cK$ coincide. 
\end{lemma}
\begin{proof}
  We have to prove the inclusion
  \begin{equation*}
    \cat{ho}_*(D) \supseteq \cat{ho}_{*s}(D),
  \end{equation*}
  the opposite one being obvious. To that end, we fix an automorphism $\alpha$ of $\cT\otimes \cK$.

  $\cT\otimes \cK$ fits into the extension
  \begin{equation}\label{eq:ext}
    0\to \cK'\otimes \cK \to \cT\otimes \cK \to C(\check{S}(D))\otimes \cK \to 0,
  \end{equation}
  where
  \begin{equation*}
    \cK':=\cK(H^2(\check{S}(D)))
  \end{equation*}
  is the algebra of compact operators on the Hardy space associated to $D$.
  
  \Cref{eq:ext} is preserved by $\alpha$, which thus leaves the corresponding element of the K-homology group
  \begin{equation}\label{eq:khom}
    K_1(\check{S}(D))\cong KK^1(C(\check{S}(D))\otimes \cK, \cK'\otimes \cK) \cong KK^1(C(\check{S}(D)), \cK). 
  \end{equation}
  As noted before, the rank-1 bounded symmetric domains are those of type $I_{p,1}$ (\cite[Table 1]{viv}), with respective Shilov boundaries $\bS^{2p-1}$, and \cite{bdf,ys-sph} ensure that their $\cat{ho}_*$ groups are those of degree-1 homeomorphisms. This means that
  \begin{itemize}
  \item the K-homology group \Cref{eq:khom} is isomorphic to $K_1(\bS^{2p-1})\cong \bZ$;
  \item the class of the extension \Cref{eq:ext} is a non-trivial element therein;
  \item the homeomorphism $\alpha$ of $\bS^{2p-1}$ fixes that non-trivial element and hence has degree 1.
  \end{itemize}
  It then follows from the main result of \cite{ys-sph} that $\alpha_i$ is liftable to $\cT(D_i)$, proving the claim.
\end{proof}

\begin{corollary}\label{cor:1eq}
  For bounded symmetric domains \Cref{eq:dprod} decomposing as products of rank-1 domains we have
  \begin{equation*}
    \cat{ho}_*(\cD)\rtimes \cat{sym}_{\cD} = \cat{ho}_*(D)
  \end{equation*}
\end{corollary}
\begin{proof}
This is immediate from \Cref{th:permtens,le:1easy}.   
\end{proof}

\subsection{Reconstructing higher-rank domains}\label{subse:rechigh}

We now revisit the theme broached in \Cref{subse:recirr}, of recovering a domain from its Toeplitz algebra. The material in this section allows us to extend \Cref{th:stabiso} to arbitrary domains.

\begin{theorem}\label{th:stabisogen}
  If $D_1$ and $D_2$ are two bounded symmetric domains such that $\cT(D_i)$ are stably isomorphic then $D_i$ are isomorphic.
\end{theorem}
\begin{proof}
  Decompose
  \begin{equation*}
    D_i\cong D_{i,1}\times\cdots\times D_{i,s_i}
  \end{equation*}
  into irreducible components. The argument in the proof of \Cref{cor:presfull} in fact shows that an isomorphism
  \begin{equation*}
    \cT(D_1)\otimes \cK\cong \cT(D_2)\otimes \cK
  \end{equation*}
  will in fact map the ideals
  \begin{equation}\label{eq:bullj1}
    \bullet_j(D_1)\otimes \cK \subseteq \cT(D_1)\otimes \cK
  \end{equation}
  onto their counterparts
  \begin{equation*}
    \bullet_{j'}(D_2)\otimes \cK \subseteq \cT(D_2)\otimes \cK,
  \end{equation*}
  perhaps in some permuted order; note that this line of reasoning implies $s_1=s_2$, which we did not assume a priori.

  But \Cref{eq:bullj1} are respectively isomorphic to Toeplitz algebras of the form $\cT(D_{1,j})\otimes \cK_j$ for elementary $C^*$-algebras $\cK$, and similarly for $D_2$. \Cref{th:stabiso} then implies that the Cartesian factors $D_{1,j}$ are isomorphic to $D_{2,j}$ (again, possibly upon permuting), hence the conclusion.
\end{proof}

\section{Trivial actions on spectra}\label{se:autotriv}

\begin{theorem}\label{th:trivact}
  For any bounded symmetric domain $D$ and elementary $C^*$-algebra $\cK$, an automorphism of $\cT(D)\otimes \cK$ that acts trivially on $\check{S}(D)$ acts trivially on the entire spectrum
  \begin{equation*}
   \widehat{\cT(D)}\cong \widehat{\cT(D)\otimes \cK}. 
  \end{equation*}
\end{theorem}
\begin{proof}
  We set $\cK=\bC$ throughout to keep the notation simple; the reasoning applies generally.

  Recall the filtration \Cref{eq:filt} of $\cT:=\cT(D)$, with elements of
  \begin{equation*}
    \widehat{\cI_k/\cI_{k-1}}\subset \widehat{\cT} 
  \end{equation*}
  being classes of irreducible representations $\sigma_e$ parametrized by rank-$k$ tripotents $e$. The automorphism group of $\cT$ acts by homeomorphisms on the manifold $T_k$ of rank-$k$ tripotents for $k=0..r$ (where $r$ is the rank), and by assumption that action is trivial on $T_r$. We will argue by downward induction on $k\le r$ that the action on $T_k$ is trivial, the base case having been taking care of by the hypothesis.

  For the induction step, fix an automorphism $\alpha$ of $\cT$, a $k<r$, and assume the action of $\alpha$ on $T_{k'}$ is trivial for all $k<k'\le r$. The goal is to show that for any rank-$k$ tripotent $e$, the kernel $\cI_e$ of the associated irreducible representation $\sigma_e$ is invariant under $\alpha$. 

  Consider those tripotents $f$ of rank $>k$ which dominate $e$ in the sense that there is a tripotent $e'$, orthogonal to $e$, with $f=e+e'$ (see e.g. \cite[\S 5.1]{loos}); we write $f>e$. It follows from \cite[Corollary 3.11 and discussion preceding Lemma 3.4]{up-alg} that the intersection
  \begin{equation*}
    \cK_e:=\bigcap_{f>e}\ker\sigma_f
  \end{equation*}
  is precisely the preimage of the compact operators through $\sigma_e$, and hence is $\alpha$-invariant because $\ker\sigma_f$ are. But then, since $\cI_e\subset \cK_e$ is the unique maximal ideal, it too must be $\alpha$-invariant. This concludes the (downward) induction step and the proof of the theorem.
\end{proof}

Consider, more generally, a solvable algebra $\cA$ with a filtration \Cref{eq:filt} (naturally, ending in $\cA$ rather than $\cT(D)$). It is not a given that the analogue of \Cref{th:trivact} holds, even when $\cI_k$ are characteristic. In other words, as \Cref{ex:ntriv} confirms, it is possible for automorphisms to be trivial on some of the layers
\begin{equation*}
  \widehat{\cI_k/\cI_{k-1}}\subset \widehat{\cA} 
\end{equation*}
but not others. 

\begin{example}\label{ex:ntriv}
  For some real deformation parameter $\mu$ with $|\mu|<1$ let $\cA$ be the function algebra $C(S_{\mu}U(2))$ on the compact quantum group $S_{\mu}U(2)$ studied in \cite{wor-su2}. As shown in \cite[Appendix A2]{wor-su2} and recalled in \cite[\S 1]{sheu-quant}, it fits into an exact sequence
  \begin{equation*}
    0\to C(\bS^1)\otimes \cK\to \cA\to C(\bS^1)\to 0,
  \end{equation*}
  characteristic because the right hand surjection is precisely the abelianization of $\cA$ ($\cK$ means compact operators on a countable-dimensional Hilbert space). $\cA$ is thus solvable of length 1, with
  \begin{equation*}
    \cI_0 \cong C(\bS^1)\otimes \cK\text{ and }\cI_1/\cI_0\cong C(\bS^1). 
  \end{equation*}

  By \cite[Examples 1.7.4 and 1.7.8]{nt-bk} and the description of the irreducible $\cA$-representations (e.g. \cite[Theorem 3.2]{sv}) there are one-parameter automorphism groups $\sigma_t$ and $\tau_s$ (the {\it modular} and {\it scaling} group respectively, familiar from the theory of compact quantum groups), that act
  \begin{itemize}
  \item $\sigma_t$: trivially on the ``bottom'' layer $\widehat{\cI_0}\cong \bS^1$ but not the ``top'' layer $\widehat{\cI_1/\cI_0}\cong \bS^1$;
  \item $\tau_s$: vice versa, trivially on the top layer but not the bottom one. 
  \end{itemize}
\end{example}

\appendix

\section{Numerical data on bounded symmetric domains}\label{se:a}

We gather some of the bounded-symmetric-domain numerology useful in the main text, presented here in the form that is most easily applicable in the discussion above. The information is compiled mainly from \cite{viv,sat}.

For the list of irreducible domains together with their (real) dimensions and ranks we look to \cite[Table 1]{viv}, choosing our parameter ranges so as to avoid the small-rank coincidences listed in \cite[Table 7]{viv}. With this caveat, we further list the tube and non-tube domains separately, recovering the former from \cite[Table 11]{viv} and the latter by elimination. In the tube case we have


\begin{table}[h]
  \centering
  \begin{tabular}{|c|c|c|}
    \hline
    type & $\dim_{\bR}$ & rank \\ \hline
    $I_{q,q},\ q\ge 1$ & $2q^2$ & $q$ \\ \hline
    $II_{2q},\ q\ge 3$ & $2q(2q-1)$ & $q$ \\ \hline
    $III_q,\ q\ge 2$ & $q(q+1)$ & $q$ \\ \hline
    $IV_q,\ q\ge 5$ & $2q$ & 2\\ \hline
    $VI$ & 54 & 3 \\ \hline
  \end{tabular}
  \caption{Tube-type domains}
  \label{tab:tube}
\end{table}

In particular, by direct examination we obtain

\begin{lemma}\label{le:tubedr}
  An irreducible tube-type bounded symmetric domain is uniquely determined by its rank and dimension.
  \qedhere
\end{lemma}

On the other hand, for non-tube domains the rank and dimension are not quite enough. Supplementing this information with the dimension of the Shilov boundary $\check{S}(D)$, however, will produce a complete invariant, as seen by examining \Cref{tab:ntube}.

\begin{table}[h]
  \centering
  \begin{tabular}{|c|c|c|c|}
    \hline
    type & $\dim_{\bR}$ & rank & $\dim ~\text{Shilov boundary}$\\ \hline
    $I_{p,q},\ p>q\ge 1$ & $2pq$ & $q$ & $2pq-q^2$\\ \hline
    $II_{2q+1},\ q\ge 2$ & $2q(2q+1)$ & $q$ & $2q^2+3q$\\ \hline
    $V$ & 32 & 2 & 24\\ \hline    
  \end{tabular}
  \caption{Non-tube-type domains}
  \label{tab:ntube}
\end{table}

In short:

\begin{lemma}\label{le:ntubedrd}
  An irreducible non-tube-type bounded symmetric domain is uniquely determined by its rank, dimension, and Shilov-boundary dimension.
  \qedhere
\end{lemma}

\begin{remark}
  \Cref{tab:tube} implicitly records the Shilov-boundary dimension too: by \cite[Theorem 4.9]{kw}, for tube domains $\dim\check{S}(D)$ is precisely half of $\dim_{\bR}D$.
\end{remark}

As a consequence of all of the above, we have

\begin{corollary}\label{cor:complinv}
  An irreducible bounded symmetric domain is determined by its rank, dimension, and (homeomorphism class of the) Shilov boundary. 
\end{corollary}
\begin{proof}
  This follows from \Cref{le:tubedr,le:ntubedrd}, once we recall that an irreducible bounded symmetric domain is of tube type if and only if its first Betti number is 1 (e.g. \cite[Theorem 4.11]{kw}).
\end{proof}


\bibliographystyle{plain}

\Addresses

\end{document}